\documentclass[12pt]{article}
\usepackage{mycommands}
\title{Interpolation for Restricted Tangent Bundles of General Curves}

\begin{document}
\maketitle

\begin{abstract}
Let $(C, p_1, p_2, \ldots, p_n)$ be a general marked curve of genus $g$,
and $q_1, q_2, \ldots, q_n \in \pp^r$ be a general collection of points.
We determine when there exists a nondegenerate
degree $d$ map $f \colon C \to \pp^r$
so that $f(p_i) = q_i$ for all $i$.
This is a consequence of our main theorem, which states that the
restricted tangent bundle $f^* T_{\pp^r}$ of a general curve of genus $g$,
equipped with a general degree~$d$ map $f$ to $\pp^r$,
satisfies the property of \emph{interpolation}
(i.e.\ that for a general effective divisor $D$ of any degree on $C$,
either $H^0(f^* T_{\pp^r}(-D)) = 0$ or $H^1(f^* T_{\pp^r}(-D)) = 0$).
We also prove an analogous theorem for the twist $f^* T_{\pp^r}(-1)$.
\end{abstract}

\section{Introduction}

The goal of this paper is to answer the following basic question
about incidence conditions for curves:

\begin{quest*}
Fix a general marked curve $(C, p_1, p_2, \ldots, p_n)$ of genus $g$,
and $n$ general points $q_1, q_2, \ldots, q_n \in \pp^r$.
When does there exist a nondegenerate
degree $d$ map $f \colon C \to \pp^r$
so that $f(p_i) = q_i$ for all $i$?
\end{quest*}

(An analogous question when $(C, p_1, p_2, \ldots, p_n)$
is allowed to vary in moduli was recently answered in \cite{aly}
for curves with nonspecial hyperplane section.)

In order for there to exist any nondegenerate degree $d$
maps $f \colon C \to \pp^r$, the Brill-Noether theorem \cite{bn} 
states that the \emph{Brill-Noether number}
\[\rho(d, g, r) := (r + 1)d - rg - r(r + 1)\]
must be nonnegative;
so we assume this is the case for the remainder of the paper.
In this case,
writing $\map_d(C, \pp^r)$ for the space of nondegenerate degree $d$ maps
$C \to \pp^r$,
the Brill-Noether theorem additionally gives
\[\dim \map_d(C, \pp^r) = \rho(d, g, r) + \dim \aut \pp^r = (r + 1)d - rg + r.\]
Thus, the answer to our main question can only be positive when
\begin{equation} \label{fund-ineq}
(r + 1)d - rg + r - rn \geq 0.
\end{equation}
Our main theorem will imply (as an immediate consequence)
that conversely, the answer to our main question is positive
when the above inequality holds. To state the main theorem,
we will first need to make a definition:

\begin{defi} \label{def:inter} We say that a vector bundle $\mathcal{E}$ on a curve $C$
\emph{satisfies interpolation} if it is nonspecial (i.e.\ $H^1(\mathcal{E}) = 0$), and
for a general effective divisor $D$
of any degree $d \geq 0$, either
\[H^0(\mathcal{E}(-D)) = 0 \quad \text{or} \quad H^1(\mathcal{E}(-D)) = 0.\]
For $C$ reducible, we require that
the above holds for an effective divisor $D$ which is general
in \emph{some} (not every) component of $\sym^d C$ (for each $d \geq 0$).
\end{defi}

\noindent
With this notation, our main result is:

\begin{thm} \label{main} Let $C$ be a general curve of genus $g$,
equipped with a general nondegenerate map
$f \colon C \to \pp^r$ of degree $d$. Then $f^* T_{\pp^r}$
satisfies interpolation.
\end{thm}

To see that this theorem implies a positive answer to the main
question subject to the inequality \eqref{fund-ineq},
we first note that by basic deformation theory,
the map $\map_d(C, \pp^r) \to (\pp^r)^n$,
defined via $f \mapsto (f(p_i))_{i = 1}^n$, 
is smooth at $f$ provided that $H^1(f^* T_{\pp^r}(-p_1-\cdots-p_n)) = 0$.
Note that the left-hand side of \eqref{fund-ineq}
equals $\chi(f^* T_{\pp^r})$. Consequently,
it suffices to show, for $C$ a general curve of genus $g$,
equipped with a general nondegenerate map $f \colon C \to \pp^r$,
that $H^1(f^* T_{\pp^r}(-p_1-\cdots-p_n))$ vanishes
whenever
$\chi(f^* T_{\pp^r}(-p_1-\cdots-p_n)) \geq 0$.
But this is immediate provided that $f^* T_{\pp^r}$ satisfies interpolation.
We have thus shown that our main theorem implies:

\begin{cor} \label{main-cor}
Fix a general marked curve $(C, p_1, p_2, \ldots, p_n)$ of genus $g$,
and $n$ general points $q_1, q_2, \ldots, q_n \in \pp^r$.
There exists a nondegenerate
degree $d$ map $f \colon C \to \pp^r$
so that $f(p_i) = q_i$ for all $i$ if and only if
$\rho(d, g, r) \geq 0$ and
\[(r + 1) d - rg + r - rn \geq 0.\]
\end{cor}

Additionally, we prove a similar theorem for the twist
of the tangent bundle; this answers a similar question,
where $d$ of the $n$ points are constrained to lie on a hyperplane:

\begin{thm} \label{main-1} Let $C$ be a general curve of genus $g$,
equipped with a general nondegenerate map
$f \colon C \to \pp^r$ of degree $d$. Then $f^* T_{\pp^r}(-1)$
satisfies interpolation if and only if
\[d - rg - 1 \geq 0.\]
\end{thm}

To see the ``only if'' direction, first note that
since $C$ is nondegenerate,
the restriction map $H^0(T_{\pp^r}(-1)) \to H^0(f^* T_{\pp^r}(-1))$
is injective, so $\dim H^0(f^* T_{\pp^r}(-1)) \geq \dim H^0(T_{\pp^r}(-1)) = r + 1$.
If $d - rg - 1 < 0$, then
$\chi(f^* T_{\pp^r}(-1)) = d - rg + r < r + 1$; consequently
$H^1(f^* T_{\pp^r}(-1)) \neq 0$,
so $f^* T_{\pp^r}(-1)$ does not satisfy interpolation.
For the remainder of the paper, we will therefore assume
$d - rg - 1 \geq 0$ in Theorem~\ref{main-1}.
As before, Theorem~\ref{main-1} gives the ``if'' direction of:

\begin{cor} \label{main-1-cor}
Fix a general marked curve $(C, p_1, p_2, \ldots, p_n)$ of genus $g$,
and $n$ points $q_1, q_2, \ldots, \linebreak[0] q_n \in \pp^r$ which are general subject
to the constraint that $q_1, q_2, \ldots, q_d$ lie on a hyperplane $H$
(for $d \leq n$).
There exists a nondegenerate
degree $d$ map $f \colon C \to \pp^r$
so that $f(p_i) = q_i$ for all $i$ if and only if $\rho(d, g, r) \geq 0$ and
\[(r + 1) d - rg + r - rn \geq 0 \quad \text{and} \quad d - rg - 1 \geq 0.\]
\end{cor}

For the ``only if'' direction, we first compute the dimension of
the space of degree~$d$ maps
$\map_d((C, p_1 \cup p_2 \cup \cdots \cup p_d), (\pp^r, H))$,
from $C$ to $\pp^r$, which send $p_1 \cup p_2 \cup \cdots \cup p_d$ to $H$.
Choose coordinates $[x_1 : x_2 : \cdots : x_{r + 1}]$ on $\pp^r$
so $H$ is given by $x_{r + 1} = 0$.
Then any such map is given by $r + 1$ sections
$[s_1 : s_2 : \cdots : s_r : 1]$
of $\oo_C(p_1 + p_2 + \cdots + p_d)$,
where we write
$s_{r + 1} = 1 \in H^0(\oo_C) \subset H^0(\oo_C(p_1 + \cdots + p_d))$
for the constant section.
By taking $s_1, s_2, \ldots, s_r$, we identify
$\map_d((C, p_1 \cup p_2 \cup \cdots \cup p_d), (\pp^r, H))$
as an open subset of $H^0(\oo_C(p_1 + \cdots + p_d))^r$,
so its dimension is $r(d + 1 - g)$.
In order for $f \mapsto (f(p_i))_{i = 1}^n$ to dominate
$H^d \times (\pp^r)^{n - d}$ we must therefore also have
\[r(d + 1 - g) \geq (r - 1)d + r(n - d) \quad \Leftrightarrow \quad (r + 1)d - rg + r - rn \geq 0.\]
In addition, $f \mapsto (f(p_i))_{i = 1}^d$
must dominate $H^d$. But the fibers of this map have
a free action of the $(r + 1)$-dimensional group $\stab_H (\aut \pp^r)$;
we must therefore also have
\[r(d + 1 - g) \geq (r - 1)d + (r + 1) \quad \Leftrightarrow \quad d - rg - 1 \geq 0.\]

\begin{rem}
For $k \geq 2$, the bundles $f^* T_{\pp^r}(-k)$ almost never satisfy interpolation.
Namely, in the setting of Theorem~\ref{main}, the bundle 
$f^* T_{\pp^r}(-2)$ satisfies interpolation if and only if either $(g, r) = (0, 1)$ or
$(d, g, r) = (2, 0, 2)$;
and $f^* T_{\pp^r}(-k)$ never satisfies interpolation for $k \geq 3$.

For $r = 1$, this can be seen by observing that $f^* T_{\pp^1}(-k) \simeq \oo_C(2 - k)$
is a line bundle, so satisfies interpolation if and only if it is nonspecial
(c.f.\ Proposition~4.7 of~\cite{aly}); as $k \geq 2$,
this only happens if $k = 2$ and $C$ is rational ($g = 0$).

In general, if $f^* T_{\pp^r}(-k)$ satisfies interpolation,
$\chi(f^* T_{\pp^r}(-k)) = (r + 1)d - rg + r - krd \geq 0$.
When $r \geq 2$ and $k \geq 2$, this is only satisfied for $(d, g, r, k) = (2, 0, 2, 2)$;
conversely, one may easily check that $f^* T_{\pp^2}(-2)$ satisfies interpolation
for $f \colon \pp^1 \to \pp^2$ a general degree~$2$ map (for instance by combining Theorem~\ref{main}
with Proposition~4.12 of~\cite{aly}).
\end{rem}

The remainder of the paper will be devoted to the proof of Theorem~\ref{main}
and Theorem~\ref{main-1} using inductive degeneration.
In Section~\ref{sec:prelim}, we begin by explaining how to use degeneration
to approach the main theorems, and how to work with the condition of interpolation on reducible curves.
Then in Section~\ref{sec:rational}, we prove the main theorems for rational
curves by inductively degenerating the curve $C$ to a union
$D \cup L$, where $L$ is a $1$-secant line to $D$.
Finally in Section~\ref{sec:arbitrary}, we prove the main theorems
for arbitrary genus
by degenerating the curve $C$ to a union
$D \cup R$,
where $R$ is a rational normal curve meeting $D$ at $1 \leq s \leq r + 2$
points.

\subsection*{Acknowledgements}

The author would like to thank Joe Harris for
his guidance throughout this research, as well as
Atanas Atanasov and David Yang for helpful conversations.
The author would also like to
acknowledge the generous
support both of the Fannie and John Hertz Foundation,
and of the Department of Defense
(NDSEG fellowship).

\section{Preliminaries \label{sec:prelim}}

\begin{lm} \label{can-specialize}
The locus of $f \colon C \to \pp^r$ in $\bar{M}_g(\pp^r, d)$
for which $f^* T_{\pp^r}$ (respectively $f^* T_{\pp^r}(-1)$)
satisfies interpolation is open.
Moreover, every component of this locus dominates $\bar{M}_g$.

In particular, to prove Theorem~\ref{main} (respectively Theorem~\ref{main-1})
for curves of degree $d$ and genus $g$,
it suffices to exhibit
one possibly-singular nondegenerate $f \colon C \to \pp^r$ of degree $d$ and genus $g$,
for which $f^* T_{\pp^r}$ (respectively $f^* T_{\pp^r}(-1)$) satisfies interpolation.
\end{lm}
\begin{proof}
Since the vanishing of cohomology groups is an open condition,
it follows that interpolation is an open condition as well.
(For a more careful proof, see Theorem~5.8 of \cite{nasko}.)

If $f^* T_{\pp^r}$ (respectively $f^* T_{\pp^r}(-1)$) satisfies interpolation,
then in particular $H^1(f^* T_{\pp^r}) = 0$ (respectively
$H^1(f^* T_{\pp^r}(-1)) = 0$).
 Since 
$H^1(f^* T_{\pp^r}(-1)) = 0$ implies $H^1(f^* T_{\pp^r}) = 0$,
we know either way that
$H^1(f^* T_{\pp^r}) = 0$.
This completes the proof
as the obstruction to smoothness
of $\bar{M}_g(\pp^r, d) \to \bar{M}_g$
lies in $H^1(f^* T_{\pp^r})$.
\end{proof}

In order to work with reducible curves, it will be helpful to
introduce the following somewhat more general variant on interpolation:

\begin{defi} \label{T:interpolation}
Let $\mathcal{E}$ be a rank $n$ vector bundle over a curve $C$.
We say that a subspace of sections
$V \subseteq H^0(\mathcal{E})$ satisfies interpolation if
$\mathcal{E}$ is nonspecial and, for every $d \geq 0$,
there exists an effective Cartier divisor $D$ of degree $d$ such that
  \[
  \dim \left( V \cap H^0\left( \mathcal{E}(-D) \right)\right) = \max\{0, \dim V - d n \}.
  \]
\end{defi}

We now state for the reader
the elementary properties of interpolation that
we shall need from \cite{aly} and \cite{nasko}.

\begin{lm} \label{vb-sections}
A vector bundle
$\mathcal{E}$ satisfies interpolation if and only if its full space of sections $V = H^0(\mathcal{E}) \subseteq H^0(\mathcal{E})$
satisfies interpolation. 
\end{lm}
\begin{proof}[Proof sketch (see Proposition~4.5, c.f.\ also Definition~4.1, of~\cite{aly} for a complete proof).]
The result follows by examining the long exact sequence
\[0 \to H^0(\mathcal{E}(-D)) \to H^0(\mathcal{E}) \to H^0(\mathcal{E}|_D) \to H^1(\mathcal{E}(-D)) \to H^1(\mathcal{E}) = 0. \qedhere\]
\end{proof}

\begin{lm} \label{trivglue}
Let $\mathcal{E}$ be a vector bundle on a reducible curve
$X \cup Y$, and $D$ be an effective divisor on $X$ disjoint from $X \cap Y$.
Assume that
\[H^0(\mathcal{E}|_X(-D - X \cap Y)) = 0.\]
Let
\begin{gather*}
\ev_X \colon H^0(\mathcal{E}|_X) \longrightarrow H^0(\mathcal{E}|_{X \cap Y}) \\
\ev_Y \colon H^0(\mathcal{E}|_Y) \longrightarrow H^0(\mathcal{E}|_{X \cap Y})
\end{gather*}
denote the natural evaluation morphisms.
Then $\mathcal{E}$ satisfies interpolation provided that
\[V = \ev_Y^{-1}(\ev_X(H^0(\mathcal{E}|_X(-D)))) \subseteq H^0(\mathcal{E}|_Y)\]
satisfies interpolation and has dimension
\[\chi(\mathcal{E}|_Y) + \chi(\mathcal{E}|_X(-D - X \cap Y)).\]
\end{lm}
\begin{proof}[Proof sketch (see Proposition~8.1 of \cite{aly} for a complete proof).]
As $H^0(\mathcal{E}|_X(-D - X \cap Y)) = 0$, restriction to $Y$
gives an isomorphism $H^0(\mathcal{E}(-D)) \simeq V$.
Also, the final dimension statement implies
$H^1(\mathcal{E}(-D)) = 0$.
Therefore $\mathcal{E}(-D)$, and hence $\mathcal{E}$, satisfies interpolation.
\end{proof}

\begin{lm} \label{genmod}
Let $\mathcal{E}$ be a vector bundle on an irreducible curve $C$, and
$p \in C_{\text{sm}}$ be a general point. If $\mathcal{E}$ satisfies
interpolation, and $\Lambda \subseteq \mathcal{E}|_p$
is a general subspace of any dimension, then
\[\{\sigma \in H^0(\mathcal{E}) : \sigma|_p \in \Lambda\} \subseteq H^0(\mathcal{E})\]
satisfies interpolation and has dimension $\max\{0, \chi(\mathcal{E}) - \codim \Lambda\}$.
\end{lm}
\begin{proof}[Proof sketch (see Theorem~8.1, c.f.\ also Section~3, of~\cite{nasko} for a complete proof).]
Let $n = \rk \mathcal{E}$, and $D$ be a general effective divisor of any degree $d \geq 0$.
Since $\mathcal{E}$ satisfies interpolation,
\begin{align*}
\dim H^0(\mathcal{E}(-D)) &= \max\{0, \chi(\mathcal{E}) - dn\} \\
\dim H^0(\mathcal{E}(-D - p)) &= \max\{0, \chi(\mathcal{E}) - dn - n\}.
\end{align*}
These inequalities imply the restriction map $H^0(\mathcal{E}(-D)) \to \mathcal{E}|_p$
is either injective or surjective, which (since $\Lambda$ is general)
in turn implies the composition
$H^0(\mathcal{E}(-D)) \to \mathcal{E}|_p \to \mathcal{E}|_p / \Lambda$
is either injective or surjective.
Together with the first of the above equalities, this implies
\[\dim \{\sigma \in H^0(\mathcal{E}(-D)) : \sigma|_p \in \Lambda\} = \max\{0, \chi(\mathcal{E}) - \codim \Lambda - dn\}. \qedhere\]
\end{proof}

\section{Rational Curves \label{sec:rational}}

In this section, we prove Theorem~\ref{main} in the case $g = 0$.

\begin{prop} \label{prop:line} Let $L \subseteq \pp^r$ be a line. Then
$T_{\pp^r}|_L \simeq \oo_L(2) \oplus \oo_L(1)^{r - 1}$,
where the $\oo_L(2)$ summand comes from the inclusion $T_L \hookrightarrow T_{\pp^r}|_L$;
in particular, $T_{\pp^r}|_L$ and $T_{\pp^r}|_L(-1)$ satisfy interpolation.
\end{prop}
\begin{proof}
The first assertion follows from the exact sequence
\[0 \to T_L \simeq \oo_L(2) \to T_{\pp^r}|_L \to N_{L/\pp^r} \simeq \oo_L(1)^{r - 1} \to 0.\]
(We have $N_{L/\pp^r} \simeq \oo_L(1)^{r - 1}$ since $L$
is the complete intersection of $r - 1$ hyperplanes.)

The second assertion follows from the first
by inspection (or alternatively using Proposition~3.11 of~\cite{nasko}).
\end{proof}

\begin{prop} \label{prop:rat}
Let $f \colon \pp^1 \to \pp^r$ be a general degree $d$ map
(allowed to be degenerate if $d < r$).
Then $f^*T_{\pp^r}(-1)$ --- and consequently $f^* T_{\pp^r}$ --- satisfies interpolation.
\end{prop}
\begin{proof}
To show $f^* T_{\pp^r}(-1)$ satisfies interpolation,
we argue by induction on the degree $d$ of $f$;
the base case $d = 1$ is given by Proposition~\ref{prop:line}.
For the inductive step $d \geq 2$,
we degenerate $f \colon \pp^1 \to \pp^r$
to a map $g \colon D \cup_p L \to \pp^r$ from a two-component reducible
rational curve; write $g_D = g|_D$ and $g_L = g|_L$.
Assume that $\deg g_D = d - 1$ and
$\deg g_L = 1$.
By Lemma~\ref{can-specialize},
it suffices to show $g^* T_{\pp^r}(-1)$ satisfies interpolation.
Write
\begin{gather*}
\ev_D \colon H^0(g_D^* T_{\pp^r}(-1)) \to H^0(T_{\pp^r}(-1)|_p) \\
\ev_L \colon H^0(g_L^* T_{\pp^r}(-1)) \to H^0(T_{\pp^r}(-1)|_p)
\end{gather*}
for the natural evaluation morphisms. Pick a point $x \in L \smallsetminus p$.
Then by Lemma~\ref{trivglue}, it suffices to show
\[V = \ev_D^{-1}(\ev_L(H^0(g_L^* T_{\pp^r}(-1)(-x)))) \subseteq H^0(g_D^* T_{\pp^r}(-1))\]
satisfies interpolation and has dimension
\[\chi(g_D^* T_{\pp^r}(-1)) + \chi(g_L^* T_{\pp^r}(-1)(-x-p)) = \chi(g_D^* T_{\pp^r}(-1)) - (r - 1).\]
Using the description of $g_L^* T_{\pp^r}$ from Proposition~\ref{prop:line},
\[V = \{\sigma \in H^0(g_D^* T_{\pp^r}(-1)) : \sigma|_p \in T_L(-1)|_p\}.\]
Since $T_L(-1)|_p \subseteq T_{\pp^r}(-1)|_p$ is a general subspace of codimension $r - 1$,
Lemma~\ref{genmod} implies $V$ satisfies interpolation
and has dimension $\max\{0, \chi(g_D^* T_{\pp^r}(-1)) - (r - 1)\}$.
It thus suffices to note that
\[\chi(g_D^* T_{\pp^r}(-1)) - (r - 1) = d + 1 \geq 0.\]

By inspection (or alternatively using Proposition~4.11 of~\cite{aly}),
interpolation for $f^*T_{\pp^r}(-1)$ implies interpolation for $f^* T_{\pp^r}$.
\end{proof}

\section{Curves of Higher Genus \label{sec:arbitrary}}

\begin{lm} \label{lm:higher} 
Let $\mathcal{E}$ be a vector bundle of rank~$n$ on a reducible nodal curve
$X \cup Y$, such that $X \cap Y$ is a general collection of $k$ points
on $X$ (relative to $\mathcal{E}|_X$).
Suppose that $\mathcal{E}|_X$ and $\mathcal{E}|_Y$ satisfy interpolation.
If $\chi(\mathcal{E}|_X) \equiv 0$ mod $n$ and $\chi(\mathcal{E}|_X) \geq nk$,
then $\mathcal{E}$ satisfies interpolation.
\end{lm}
\begin{proof}
Let $D$ be a general divisor on $X$ (in particular disjoint from $X \cap Y$)
of degree $\frac{\chi(\mathcal{E}|_X) - nk}{n}$. Write
\begin{gather*}
\ev_X \colon H^0(\mathcal{E}|_X) \longrightarrow H^0(\mathcal{E}|_{X \cap Y}) \\
\ev_Y \colon H^0(\mathcal{E}|_Y) \longrightarrow H^0(\mathcal{E}|_{X \cap Y})
\end{gather*}
for the natural evaluation morphisms, and let
\[V = \ev_Y^{-1}(\ev_X(H^0(\mathcal{E}|_X(-D)))) \subseteq H^0(\mathcal{E}|_Y).\]

Since $\mathcal{E}|_X$ satisfies interpolation
and $\chi(\mathcal{E}|_X(-D)) = nk$, we conclude
$\ev_X$ is an isomorphism when restricted to $H^0(\mathcal{E}|_X(-D))$.
In particular, $H^0(\mathcal{E}|_X(-D - X \cap Y)) = 0$, and
$V = H^0(\mathcal{E}|_Y)$ is the full space of sections.
Because $\mathcal{E}|_Y$ satisfies interpolation,
we conclude $V$ satisfies interpolation and has dimension
$\chi(\mathcal{E}|_Y) = \chi(\mathcal{E}|_Y) + \chi(\mathcal{E}|_X(-D - X \cap Y))$.
This implies via Lemma~\ref{trivglue} that $\mathcal{E}$ satisfies interpolation.
\end{proof}

\begin{proof}[Proof of Theorem~\ref{main}.]
We argue by induction on $g$; the base case $g = 0$
is given by Proposition~\ref{prop:rat}.
For the inductive step $g \geq 1$,
we let
\[(s, g', d') = \begin{cases}
(r + 2, g - r - 1, d - r) & \text{if $g \geq r + 1$;} \\
(g + 1, 0, d - r) & \text{otherwise.}
\end{cases}\]
By construction, $g' \geq 0$ and $1 \leq s \leq r + 2$.
Moreover, since $\rho(d, g, r) \geq 0$, we have either
$\rho(d', g', r) \geq 0$, or $g' = 0$ and $d' \geq 1$;
and in the second case $d' \geq s - 1$.

We now degenerate $f \colon C \to \pp^r$
to a map $g \colon D \cup_\Gamma \pp^1 \to \pp^r$ from a two-component reducible
curve; write $g_D = g|_D$ and $g_{\pp^1} = g|_{\pp^1}$.
By the above, we may take $D$ to be a general curve of genus $g'$,
and $\Gamma$ a collection of $s$ points general on both $D$ and $\pp^1$;
we also take $g_D$ and $g_{\pp^1}$ to be general maps
of degrees $d'$ and $r$ respectively (composing with an automorphism
of $\pp^r$ so that $g_D(\Gamma) = g_{\pp^1}(\Gamma)$ --- which 
exists since $\aut \pp^r$ acts $(r + 2)$-transitively on
points in linear general position;
note that $s \leq r + 2$, and $g_{\pp^1}$ is nondegenerate,
while $g_D$ spans at least a $\pp^{\min(r, s - 1)}$).
By Lemma~\ref{can-specialize}, it suffices to show $g^* T_{\pp^r}$
satisfies interpolation.

By induction (and direct application of Proposition~\ref{prop:rat}
in the case $g' = 0$ and $d' \geq 1$),
we know $g_D^* T_{\pp^r}$ and $g_{\pp^1}^* T_{\pp^r}$
satisfy interpolation.
Moreover, $\chi(g_{\pp^1}^* T_{\pp^r}) = r(r + 2)$
is a multiple of $r$ which is at least $rs$.
Lemma~\ref{lm:higher} thus yields the desired conclusion.
\end{proof}

\begin{proof}[Proof of Theorem~\ref{main-1}.]
We argue by induction on $g$; the base case $g = 0$
is given by Proposition~\ref{prop:rat}.
For the inductive step $g \geq 1$,
we write $g' = g - 1 \geq 0$ and $d' = d - r$.
Since $d - rg - 1 \geq 0$, we have
$d' - rg' - 1 \geq 0$. This in turn implies
either $\rho(d', g', r) \geq 0$, or 
$g' = 0$ and $d' \geq 1$.

We now degenerate $f \colon C \to \pp^r$
to a map $g \colon D \cup_\Gamma \pp^1 \to \pp^r$ from a two-component reducible
curve; write $g_D = g|_D$ and $g_{\pp^1} = g|_{\pp^1}$.
By the above, we may take $D$ to be a general curve of genus $g'$,
and $\Gamma$ a collection of $2$ points general on both $D$ and $\pp^1$;
we also take $g_D$ and $g_{\pp^1}$ to be general maps
of degrees $d'$ and $r$ respectively (composing with an automorphism
of $\pp^r$ so that $g_D(\Gamma) = g_{\pp^1}(\Gamma)$ --- which exists
since $\aut \pp^r$ acts $2$-transitively).
By Lemma~\ref{can-specialize}, it suffices to show $g^* T_{\pp^r}(-1)$
satisfies interpolation.

By induction (and direct application of Proposition~\ref{prop:rat}
in the case $g' = 0$ and $d' \geq 1$),
we know $g_D^* T_{\pp^r}(-1)$ and $g_{\pp^1}^* T_{\pp^r}(-1)$
satisfy interpolation.
Moreover, $\chi(g_{\pp^1}^* T_{\pp^r}(-1)) = 2r$.
Lemma~\ref{lm:higher} thus yields the desired conclusion.
\end{proof}

\bibliographystyle{amsplain.bst}
\bibliography{mrcbib}

\end{document}